\documentclass[10pt]{amsart}

\author{Daniel Thompson}
\usepackage[english]{babel}
\usepackage{amsmath}
\usepackage{amsthm}
\usepackage{amssymb}
\usepackage[backend=bibtex,style=alphabetic]{biblatex}
\usepackage[T1]{fontenc}
\usepackage{enumerate}
\usepackage{microtype}
\usepackage[none]{hyphenat}
\usepackage{setspace}

\bibliography{biblio.bib}

\newtheorem{theorem}{Theorem}[section]
\newtheorem{corollary}[theorem]{Corollary}
\newtheorem{lemma}[theorem]{Lemma}
\theoremstyle{definition}
\newtheorem{remark}[theorem]{Remark}

\newtheorem{example}[theorem]{Example}

\begin{document}

\title{A generalisation of Dickson's commutative division algebras}

\begin{abstract}
Dickson's commutative semifields are an important class of finite division algebras. We generalise Dickson's construction of commutative division algebras by doubling both finite field extensions and central simple algebras and not restricting us to the classical setup where a cyclic field extension is taken. The latter case yields algebras which are no longer commutative nor associative. Conditions for when the algebras are division algebras are given that canonically generalise the classical ones known up to now. We investigate when we obtain non-isomorphic algebras and compute all the automorphisms, including the structure of the automorphism group in some cases.
\end{abstract}

\maketitle

\section*{Introduction}

Dickson's commutative division algebras \cite{Dic} have been widely studied over finite fields as they yield proper finite semifields of even dimension: For any choice of $c\in K\setminus K^2$ and $\sigma\in Aut_F(K)$ not equal to the identity, $K\oplus K$ equipped with the multiplication $$(u,v)(x,y)=(ux+c\sigma(vy),ux+vy)$$ is a division algebra over $F$ when $F$ is a finite field. This construction was investigated more generally in two papers by Burmester where $K$ is a cyclic field extension of degree $n$ over a field of characteristic not 2 \cite{Bur, Bur2}, producing $2n$-dimensional unital algebras over $F$. Further, Dickson \cite{Dic} and Burmester gave a necessary and sufficient condition for when the algebras constructed this way are division.\\
Dickson's commutative division algebras also appear as a special case of a family of finite semifields constructed by Knuth \cite{Knu}: A subfield $L$ of a semifield $S$ is called a \textit{weak nucleus} if $x(yz)-(xy)z=0$, whenever two of $x,y,z$ lie in $L$. Knuth produced conditions to determine all isotopism classes of finite semifields which are quadratic over their weak nucleus; Dickson's semifields are the only commutative semifields of this type. Isotopism classes of commutative Dickson semifields were also treated in Burmester's paper \cite{Bur2} and more recently in some work by Hui, Tai and Wong \cite{HTW}.\\
In this paper, we generalise Dickson's doubling process by first doubling a finite (not necessarily cyclic) field extension $K/F$ and then a central simple associative algebra $B$ over $F$. As $B$ is not commutative, in this last setup we have more options for a possible generalisation of the multiplication given in Dickson's construction. For instance, we may define the multiplication on the $F$-vector space $B\oplus B$ as $$(u,v)(x,y)=(ux+c\sigma(vy),uy+vx)$$ for some $c\in B^{\times}$ and non-trivial $\sigma\in Aut_F(B)$, but we can also define a multiplication by putting $c$ in the middle, i.e. $$(u,v)(x,y)=(ux+\sigma(v)c\sigma(y),uy+vx),$$ or by putting $c$ on the right-hand side:  $$(u,v)(x,y)=(ux+\sigma(vy)c,uy+vx).$$ Clearly, the unital $F$-algebras we obtain this way are no longer commutative.\\
After preliminary results and definitions in Section 1, the doubling of a finite field extension is investigated in Section 2. We find multiple conditions for when we obtain division algebras this way and consider when our algebras are isomorphic. We also examine their automorphisms and determine their automorphism groups. Section 3 looks at what happens when we construct algebras starting with a central simple algebra $B$ over $F$ and employs several canonical generalisations of Dickson's doubling process. Again we examine both isomorphisms and automorphisms of these algebras and determine the size of their automorphism groups. Most importantly, we investigate when the algebras we obtained this way are division algebras. The results of this paper are part of the author's PhD thesis written under the supervision of Dr S. Pumpl\"{u}n.

\section{Definitions and preliminary results}

In the following, let $F$ be a field. We will define an $F$-algebra $A$ as a finite dimensional $F$-vector space equipped with a (not necessarily associative) bilinear map $A\times A\to A$ which is the multiplication of the algebra. $A$ is a \textit{division algebra} if for all nonzero $a\in A$ the maps $L_a:A\to A$, $x\mapsto ax$, and $R_a:A\to A$, $x\mapsto xa$, are bijective maps. As $A$ is finite dimensional, $A$ is a division algebra if and only if there are no zero divisors \cite{Sch}.\\
The \textit{associator} of $x,y,z\in A$ is defined to be $[x,y,z]:=(xy)z-x(yz).$ Define the \textit{left, middle and right nuclei} of $A$ as $Nuc_l(A):=\lbrace x\in A \mid [x,A,A]=0\rbrace,$ $Nuc_m(A):=\lbrace x\in A \mid [A,x,A]=0\rbrace,$ and $Nuc_r(A):=\lbrace x\in A \mid [A,A,x]=0\rbrace.$ The left, middle and right nuclei are associative subalgebras of $A$. Their intersection $Nuc(A):=\lbrace x\in A \mid [x,A,A]=[A,x,A]=[A,A,x]=0\rbrace$ is the \textit{nucleus} of $A$. The \textit{commuter} of $A$ is the set of elements which commute with every other element, $Comm(A):=\lbrace x\in A\mid xy=yx \:\forall y\in A\rbrace.$ The \textit{center} of $A$ is given by the intersection of $Nuc(A)$ and $Comm(A)$, $Z(A):=\lbrace x\in Nuc(A)\mid xy=yx\: \forall y\in A\rbrace.$ For two algebras $A$ and $B$, any isomorphism $f:A\to B$ maps $Nuc(A)$ isomorphically onto $Nuc(B)$. An algebra $A$ is \textit{unital} if there exists an element $1_A\in A$ such that $x1_A=1_Ax=x$ for all $x\in A$.\\
A \textit{form of degree d} over $F$ is a map $N:A\to F$ such that $N(ax)=a^dN(x)$ for all $a\in F$, $x\in A$ and such that the map $\theta:A\times...\times A\to F$ defined by $$\theta(x_1,x_2,...,x_d)=\frac{1}{d!}\sum_{1\leq i_1<...<i_l\leq d}(-1)^{d-1}N(x_{i_1}+...+x_{i_l})$$ ($1\leq l\leq d$) is a $d$-linear form over $F$. A \textit{d-linear form} over $F$ is an $F$-multilinear map $\theta:A\times...\times A\to F$ ($d$ copies) such that $\theta(x_1,x_2,...,x_d)$ is invariant under all permutations of its variables. A form $N:A\to F$ of degree $d$ is called \textit{multiplicative} if $N(xy)=N(x)N(y)$ for all $x,y\in A$ and \textit{nondegenerate} if we have $N(x)=0$ if and only if $x=0$. Note that if $N:A\to F$ is a nondegenerate multiplicative form and $A$ is a unital algebra, it follows that $N(1_A)=1_F$. Every central simple algebra of degree $d$ admits a uniquely determined nondegenerate multiplicative form of degree $d$, called the \textit{norm} of the algebra.\\

\section{Commutative Dickson algebras over any base field}\label{GeneralisingDicksonAlgebras}

\subsection{The construction process} Let $K$ be a finite field extension of $F$ of degree $n$. For some $c\in K^{\times}$ and $\sigma\in Aut_F(K)$, we define a multiplication on the $F$-vector space $K\oplus K$ by $$(u,v)(x,y)=(ux+c\sigma(vy),uy+vx)$$ for all $u,v,x,y\in K$. This makes $K\oplus K$ a unital nonassociative ring which we denote by $D(K,\sigma,c)$. Note that $D(K,id,c)$ is isomorphic to a quadratic field extension of $K$ when $c\in K\setminus K^2$ and that $D(K,id,c)\cong K\times K$ when $c\in K^{\times 2}$. Due to this, we will only consider $\sigma\neq id$. Note that $F$ is canonically embedded into $D(K,\sigma,c)$ via the map $F\mapsto F\oplus 0$. Similarly, we will denote any subalgebras of the form $E\oplus 0$ simply by $E$.

Clearly $D=D(K,\sigma,c)$ is commutative. Over finite fields, it is known that when $\sigma\neq id$, then $Nuc_l(D)=Nuc_r(D)=Fix(\sigma)$ and $Nuc_m(D)=K$ \cite[p.126]{Cor}. This is also true for any arbitrary field and is easily checked:

\begin{theorem}\label{CDSNuclei}
Let $D=D(K,\sigma,c)$ with $\sigma\in Aut_F(K)$ a non-trivial automorphism. Then we have $Nuc_l(D)=Nuc_r(D)=Fix(\sigma)$ and $Nuc_m(D)=K$. In particular, this yields $Nuc(D)=Fix(\sigma)$ and $Z(D)=Fix(\sigma)$.
\end{theorem}

A nonassociative ring is always an algebra over its centre, so $D(K,\sigma,c)$ is an algebra over $Fix(\sigma)$. However, as $F\subset Fix(\sigma)$ we can also view $D(K,\sigma,c)$ as an algebra over $F$ of dimension $2n$.

Clearly all subfields $E$ of $K$ are subalgebras of $D(K,\sigma, c)$. Additionally, if $E$ be a subfield of $K$ such that $c\in E^{\times}$ and $\sigma\!\mid_E \,\in Aut_F(E)$, then $D(E,\sigma\!\mid_E,c)$ is a subalgebra of $D(K,\sigma,c)$. Moreover, if $L=Fix(\sigma)$ and $c\in L^{\times}$, then $L\oplus L$ is an associative subalgebra of $D(K,\sigma,c)$.\\

\subsection{Division algebras}
Dickson \cite{Dic} gave a sufficient condition for $D(K,\sigma,c)$ to be a nonassociative division algebra when $F$ is an infinite field and $K/F$ is a cyclic extension. Burmester further showed this was also a necessary condition \cite[Theorem 1]{Bur}. If we assume $K/Fix(\sigma)$ is cyclic, \cite[Theorem 1]{Bur} extends naturally to our construction: 

\begin{theorem}
Let $F$ be an infinite field and $L=Fix(\sigma)$. If $Aut_L(K)=\langle\sigma\rangle$, then $D(K,\sigma, c)$ is a division algebra over $F$ if and only if $N_{K/L}(c)\neq N_{K/L}(a^2)$ for all $a\in K$.
\end{theorem}

The proof is analogous to the proof of \cite[Theorem 1]{Bur}. As it uses \cite[Theorem 5, p.200]{A2018}, we require that $F$ is not a finite field.

If $K/Fix(\sigma)$ is not a cyclic extension, this result does not necessarily hold. However, we can directly compute a different necessary and sufficient condition for $D(K,\sigma,c)$ to be a division algebra:

\begin{theorem}\label{CDSDivision}
$D(K,\sigma,c)$ is a division algebra if and only if $c\neq r^2s\sigma(s)^{-1}t^{-1}\sigma(t)^{-1}$ for all $r,s,t\in K^{\times}$.
\end{theorem}
\begin{proof}
Suppose $D(K,\sigma,c)$ is not a division algebra. Then there exist nonzero elements $(u,v),(x,y)\in K\oplus K$ such that $(u,v)(x,y)=(0,0).$ This is equivalent to the simultaneous equations \begin{align}
ux+c\sigma(vy)=&0, \label{CDSDivision1}\\
uy+vx=&0. \label{CDSDivision2}
\end{align}
If $v=0$, (\ref{CDSDivision2}) becomes $uy=0$, so either $u=0$ or $y=0$. However, $u$ must be non-zero, else $(u,v)=(0,0)$ which is a contradiction, so we must have $y=0$. Additionally, (\ref{CDSDivision1}) gives $ux=0$. As $u$ is non-zero, this implies $x=0$ and so $(x,y)=(0,0)$ which is again a contradiction.\\
So let $v\neq 0$. As $K$ is a field, we have $v^{-1}\in K$ and hence we obtain $x=-uyv^{-1}$ from (\ref{CDSDivision2}). Now if $y=0$, this implies that $x=0$ which contradicts the assumption that $(x,y)\neq(0,0)$. Substituting this into (\ref{CDSDivision1}), we get $-u^2yv^{-1}+c\sigma(vy)=0,$ which rearranges to give $c=u^2y\sigma(y)^{-1}v^{-1}\sigma(v)^{-1}$.\\
Conversely, suppose $c=r^2s\sigma(s)^{-1}t^{-1}\sigma(t)^{-1}$ for some $r,s,t\in K^{\times}$. Consider the elements $(r,t)$ and $(-rst^{-1},s)$. Both elements are nonzero but satisfy \begin{equation*}
(r,t)(-rst^{-1},s)=(-r^2st^{-1}+r^2s\sigma(s)^{-1}t^{-1}\sigma(t)^{-1}\sigma(ts),rs-rst^{-1}t)=(0,0).
\end{equation*}
Hence $D(K,\sigma,c)$ is not a division algebra.
\end{proof}

\begin{corollary}\label{CDSDivisionNorm} If $N_{K/F}(c)\neq N_{K/F}(a^2)$ for all $a\in K^{\times}$, then $D(K,\sigma,c)$ is a division algebra.
\end{corollary}
\begin{proof}
Suppose $D(K,\sigma,c)$ is not a division algebra. By Theorem \ref{CDSDivision}, there exists some $r,s,t\in K^{\times}$ such that $c=r^2s\sigma(s)^{-1}t^{-1}\sigma(t)^{-1}$. Taking norms of both sides of the equation we obtain $N_{K/F}(c)= N_{K/F}(r^2s\sigma(s)^{-1}t^{-1}\sigma(t)^{-1})$. As the norm is multiplicative and $N_{K/F}(x)=N_{K/F}(\sigma(x))$, this yields $$N_{K/F}(c)=N_{K/F}(r^2)N_{K/F}(s)N_{K/F}(s^{-1})N_{K/F}(t^{-1})^2,$$ which simplifies to $N_{K/F}(c)=N_{K/F}((rt^{-1})^2)$.
\end{proof}

\begin{corollary}\label{CDSDivisionSquare} If $c$ is a square in $K$, then $D(K,\sigma,c)$ is not a division algebra.
\end{corollary}
\begin{proof}
In the notation of Theorem \ref{CDSDivision}, let $s=t=1$. Then if $c=r^2$ for some $r\in K$, then $D(K,\sigma,c)$ is not a division algebra.
\end{proof}

\begin{remark}\begin{enumerate}[(i)]
\item Let $F=\mathbb{R}$ and $K=\mathbb{C}$. As every element of $\mathbb{C}$ is a square, we do not obtain any real division algebras by using this construction.

\item If $F$ is a finite field of characteristic 2, we also do not obtain any division algebras: again, every element is a square, so $D(K,\sigma,c)$ is not a division algebra by Corollary \ref{CDSDivisionSquare}.
\end{enumerate}
\end{remark}

\begin{example}\begin{enumerate}[(i)]
\item Let $F=\mathbb{Q}$ and $K=\mathbb{Q}(\sqrt{a})$ for some $a\in \mathbb{Q}\setminus \mathbb{Q}^2$. Then we obtain $N_{K/\mathbb{Q}}(x+y\sqrt{a})=x^2-y^2a$ for all $x,y\in \mathbb{Q}$. If we let $c=y\sqrt{a}$ for any $y\in \mathbb{Q}^{\times}$, this yields $N_{K/\mathbb{Q}}(c)=-y^2a\not\in \mathbb{Q}^2,$ so we conclude that $D(K,\sigma,c)$ is a division algebra of dimension 4 over $\mathbb{Q}$.
\item Let $F=\mathbb{Q}_p$ and $K=\mathbb{Q}_p(\alpha)$ be a quadratic field extension of $\mathbb{Q}_p$. Thus $K$ is equal to one of $\mathbb{Q}_p(\sqrt{p})$, $\mathbb{Q}_p(\sqrt{u})$ or $\mathbb{Q}_p(\sqrt{up})$, where $u\in \mathbb{Z}_p\setminus \mathbb{Z}_p^2$. If $p\equiv 1\mbox{ (mod }4)$, it follows that $-\alpha^2\not\in \mathbb{Q}_p^2$ and thus we have $N_{K/ \mathbb{Q}_p}(y\alpha)=-y^2\alpha^2\not\in \mathbb{Q}_p^2.$ Hence, $D(K,\sigma, y\alpha)$ is a division algebra of dimension 4 over $\mathbb{Q}_p$.
\end{enumerate}
\end{example}

\begin{remark}
If $F$ is a finite field of odd characteristic, we can see that Corollary \ref{CDSDivisionSquare} is also a necessary condition for $D(K,\sigma,c)$ to be a division algebra. This was originally proved in \cite[Theorem 1']{Bur} but can also be obtained as a consequence of Theorem \ref{CDSDivision}:\\

If $F=\mathbb{F}_{p^s}$ and $K=\mathbb{F}_{p^r}$ is a finite extension of $F$, it is known that $Aut_F(K)$ is cyclic of order $r/s$ and is generated by $\phi^s$, where $\phi$ is defined by the Frobenius automorphism $\phi(x)=x^p$ for all $x\in K$. Over a finite field of odd characteristic, we thus have $$\sigma(x)x=\phi^t(x)x=x^{p^{st}}x=x^{p^{st}+1}$$ for some $t\in\mathbb{Z}$. As $p$ is odd, $p^{st}+1=2n$ for some $n\in\mathbb{Z}$ and so we can write $\sigma(x)x=x^{2n}=(x^{n})^2$ for all $x\in K$. A similar argument shows that $\sigma(x)x^{-1}$ is a square for all $x\in K$. Hence over finite fields of odd characteristic, Theorem \ref{CDSDivision} yields that $D=D(K,\sigma,c)$ is a division algebra if and only if $c$ is not a square in $K$.
\end{remark}

\subsection{Isomorphisms}
For the rest of the section, we will assume that $F$ has characteristic not 2 unless stated otherwise and that $\sigma\in Aut_F(K)$ is a non-trivial automorphism. Burmester \cite{Bur} computed the isomorphisms of commutative Dickson algebras $D(K,\sigma,c)$ when $K$ is a cyclic extension of $F$. The notation originally used in \cite{Bur} differs from ours; for clarity, we rephrase his result in our notation:

\begin{theorem}[\cite{Bur}, Theorem 2]\label{BurmesterAutomorphisms}
Let $K$ be a cyclic field of degree n over $F$ and let $Aut_F(K)=\langle\sigma\rangle$. Then $D(K,\sigma^i,c)\cong D(K,\sigma^j,d)$ if and only if $i=j$, and if there exists an integer $0\leqslant k< n$ and an element $x\in K$ such that $d=x^2\sigma^k(c).$
\end{theorem}

In order to generalise this result, we first note the following two lemmas:

\begin{lemma}\label{IsomorphismRestrictsToFixFields} Let $D(K,\sigma,c)$ and $D(L,\phi, d)$ be two commutative Dickson algebras over $F$. If $Fix(\sigma)\not\cong Fix(\phi)$, then $D(K,\sigma,c)\not\cong D(L,\phi,d)$ for any choice of $c\in K^{\times}$ and $d\in L^{\times}$.
\end{lemma}
\begin{proof}
Suppose $D(K,\sigma,c)\cong D(L,\phi,c)$. As any isomorphism must map the centre of $D(K,\sigma,c)$ to the centre of $D(L,\phi,c)$, this implies $Fix(\sigma)\cong Fix(\phi)$.
\end{proof}

\begin{lemma}\label{sigmataufixfields}Let $\sigma\in Aut_F(K)$ and $\phi\in Aut_F(L)$. If there exists an $F$-isomorphism $\tau:K\to L$ such that $\tau\circ\sigma=\phi\circ\tau$, then $\tau\!\mid_{Fix(\sigma)}:Fix(\sigma)\to Fix(\phi)$ is an $F$-isomorphism.
\end{lemma}
\begin{proof}
For all $x\in Fix(\sigma)$, it follows that $$\phi\circ\tau(x)=\tau\circ\sigma(x)=\tau(x),$$ so $\tau(x)\in Fix(\phi)$. Hence we conclude that $im(\tau\mid_{Fix(\sigma)})\subseteq Fix(\phi)$. To show that in fact $im(\tau\!\mid_{Fix(\sigma)})= Fix(\phi)$, we note that for any $y\in Fix(\phi)$ there exists $x\in K$ such that $\tau(x)=y$. As $\tau(x)\in Fix(\phi)$, this implies $\tau\circ\sigma(x)=\phi\circ\tau(x)=\tau(x),$ thus $x\in Fix(\sigma)$ and it follows that $im(\tau\!\mid_{Fix(\sigma)})= Fix(\phi)$. This is sufficient to show that $\tau\mid_{Fix(\sigma)}:Fix(\sigma)\to Fix(\phi)$ is an $F$-isomorphism.
\end{proof}

\begin{theorem}\label{CDSIsomorphisms}
Let $K$ and $L$ be two finite field extensions of $F$ and $D=D(K,\sigma,c)$ and $D'=D(L,\phi,d)$ be two commutative Dickson algebras over $F$. Then $G:D\to D'$ is an isomorphism if and only if $G$ has the form $$G(x,y)=(\tau(x),\tau(y)b)$$ for some $F$-isomorphism $\tau:K\to L$ such that:\begin{enumerate}[(i)]
\item $\phi\circ\tau=\tau\circ\sigma$,
\item there exists $b\in L^{\times}$ such that $\tau(c)=d\phi(b^2)$.
\end{enumerate}
It is possible to find $b\in L^{\times}$ satisfying (i) and (ii) if and only if $\tau(c)d^{-1}$ is a square in $L^{\times}$.
\end{theorem}
\begin{proof}
Suppose $G:D\to D'$ is an $F$-isomorphism. Then $G$ maps the middle nucleus of $D$ to the middle nucleus of $D'$, so we must have $K\cong L.$ This means $G$ restricted to $K$ must be an isomorphism which maps $K$ to $L$; that is, $G\!\mid_K=\tau:K\to L$ is an isomorphism of fields and we conclude $G(x,0)=(\tau(x),0)$ for all $x\in K$. Additionally, by Lemma \ref{IsomorphismRestrictsToFixFields} we see that $Z(D)\cong Z(D')$ under $G$. Thus, it follows that $\tau$ restricted to $Fix(\sigma)$ must yield an isomorphism from $Fix(\sigma)$ to $Fix(\phi)$. Let $G(0,1)=(a,b)$ for some $a,b\in L$. This implies \begin{equation*}
G(x,y)=G(x,0)+G(0,1)G(y,0)=(\tau(x)+a\tau(y),\tau(y)b).
\end{equation*}
As $G$ is multiplicative, it follows that $G((0,1)^2)=G(0,1)^2$ which holds if and only if $(a,b)(a,b)=(\tau(c),0).$ From this, we obtain the equations $a^2+d\phi(b^2)=\tau(c)$ and $2ab=0.$
As $L$ does not have characteristic 2, this implies either $a=0$ or $b=0$. If $b=0$, then $G(x,y)=(\tau(x)+\tau(y)a,0)$ and so $G$ is not surjective. This is a contradiction, as $G$ is an isomorphism and hence is bijective by definition. Thus we obtain $a=0$ and $d\phi(b^2)=\tau(c)$.\\
Finally, as $G$ is multiplicative this yields $G(u,v)G(x,y)=G((u,v)(x,y))$ for all $u,v,x,y\in K$. Computing both sides of this equation, we get $$(\tau(ux)+d\phi(\tau(vy)b^2),\tau(uy)b+\tau(vx)b)=(\tau(ux+c\sigma(vy)),\tau(uy+vx)b)$$ for all $u,v,x,y\in K$, which implies $d\phi(\tau(vy)b^2)=\tau(c\sigma(vy)).$ After substituting the condition $d\phi(b^2)=\tau(c)$, we are left with $\phi(\tau(vy))=\tau(\sigma(vy))$ for all $v,y\in K$; that is, $\phi\circ\tau=\tau\circ\sigma$.\\
Conversely, let $G:K\oplus K\to L\oplus L$ be defined by $G(x,y)=(\tau(x),\tau(y)b)$ for some $F$-isomorphism $\tau:K\to L$ satisfying the conditions stated in the theorem above. It is easily checked that this is an $F$-linear bijective map between vector spaces. We only need to check that the map is multiplicative. Then we have $G(u,v)G(x,y)=G((u,v)(x,y))$ for all $u,v,x,y\in K$ if and only if it follows that $d\phi(\tau(vy)b^2)=\tau(c\sigma(vy)).$ As $d\phi(b^2)=\tau(c)$ and $\phi\circ\tau=\tau\circ\sigma$, this is satisfied for all $v,y\in K$. Further, by Lemma \ref{sigmataufixfields} this certainly maps the centre of $D$ to the centre of $D'$. Thus we conclude that $G:D\to D'$ is an $F$-algebra isomorphism.
\end{proof}
%

\begin{corollary}\label{CDSIsomorphismFinite}Let $D=D(K,\sigma,c)$ and $D'=D(K,\phi,d)$ be two commutative Dickson algebras over $F$. Then $G:D\to D'$ is an $F$-algebra isomorphism if and only if $G$ has the form $$G(x,y)=(\tau(x),\tau(y)b)$$ for some $\tau\in Aut_F(K)$ such that:\begin{enumerate}[(i)]
\item $\phi\circ\tau=\tau\circ\sigma$,
\item there exists $b\in K^{\times}$ such that $\tau(c)=d\phi(b^2)$.
\end{enumerate}
It is possible to find $b\in K^{\times}$ satisfying (i) and (ii) if and only if $\tau(c)d^{-1}$ is a square in $K^{\times}$.
\end{corollary}

\begin{corollary}\label{CDSIsomorphismsAbelian} Suppose $Aut_F(K)$ is an abelian group. If $\sigma\neq\phi$, then $D(K,\sigma,c)\not\cong D(K,\phi,d)$ for any choice of $c,d\in K^{\times}$.
\end{corollary}

\begin{corollary} For all $c\in K^{\times}$, we have $D(K,\sigma,c)\cong D(K,\tau\circ\sigma\circ\tau^{-1},\tau(c))$ for each $\tau\in Aut_F(K)$ 
and $D(K,\sigma,c)\cong D(K,\sigma,\sigma(b^{2})c)$ for each $b\in K^{\times}$.
\end{corollary}
\begin{proof}
This is clear employing the isomorphisms $G(x,y)=(\tau(x),\tau(y))$ and $G(x,y)=(x,b^{-1}y)$, respectively.
\end{proof}
%

When $K$ is a finite field of odd characteristic, $\tau(c)d^{-1}$ is a square if and only if either both $c$ and $d$ are squares or both are non-squares in $K$. Due to this, we obtain the following well-known result from Theorem \ref{CDSIsomorphisms}:

\begin{corollary}[\cite{Bur}, Theorem 2']\label{CDSIsomorphismClassesFinite} Let $F$ be a finite field of odd characteristic and $K$ be a finite extension of degree $n$. Let $D=D(K,\sigma,c)$ and $D'=D(K,\phi,d)$ be division algebras. Then $D\cong D'$ if and only if $\sigma=\phi$. Hence up to isomorphism, there are exactly $n$ commutative Dickson semifields of order $p^{2n}$.
\end{corollary}
\begin{proof}
First, we note that since $Aut_F(K)$ is a cyclic group, $D\cong D'$ if and only if $\sigma=\phi$ and $$G(x,y)=(\tau(x),\tau(y)b)$$ for some $\tau\in Aut_F(K)$ such that $\tau(c)=d\sigma(b^2)$ for some $b\in K^{\times}$ by Cor. \ref{CDSIsomorphismsAbelian}. As both $c,d$ are non-squares in $K$, this implies that $\tau(c)d^{-1}$ is certainly a square in the finite field $K$. Thus we can always find $b\in K$ satisfying $\sigma(b)^2=\tau(c)d^{-1}.$ Hence $D\cong D'$ if and only if $\sigma=\phi$. As each $\sigma\in Aut_F(K)$ determines a different isomorphism class of division algebras, this implies that there are $\left| Aut_F(K)\right|=n$ isomorphism classes.
\end{proof}

Over an arbitrary field however, it is possible that $D(K,\sigma,c)\not\cong D(K,\sigma,d)$ for some $c,d\in K$ as we cannot guarantee that there exists $b\in K$ such that $\sigma(b)^2=\tau(c)d^{-1}.$ Let us now consider $F=\mathbb{Q}_p$ for $p\neq 2$ as an example. We employ the following well-known result:

%

\begin{lemma} Let $K/\mathbb{Q}_p$ be a finite field extension for $p\neq 2$ with uniformizer $\pi\in \mathcal{O}_K$, where $\mathcal{O}_K$ is the valuation ring of $K$. Then $K^{\times}/(K^{\times})^2=\lbrace 1,u,\pi,u\pi\rbrace$ for some $u\in \mathcal{O}_K\setminus \mathcal{O}_K^2$.
\end{lemma}

\begin{corollary}\label{IsomorphismsoverQp}For each finite field extension $K/\mathbb{Q}_p$ such that $Aut_{\mathbb{Q}_p}(K)$ is an abelian group, there are at most $3\left| Aut_{\mathbb{Q}_p}(K)\right|$ non-isomorphic commutative Dickson division algebras of the kind $D(K,\sigma,c)$.
\end{corollary}
\begin{proof}
As in Corollary \ref{CDSIsomorphismClassesFinite}, we see that $D(K,\sigma,c)\cong D(K,\phi,d)$ if and only if $\sigma=\phi$ and there exists some $\tau\in Aut_{\mathbb{Q}_p}(K)$ and $b\in K^{\times}$ such that $\tau(c)d^{-1}=\sigma(b^2)$. Such $b\in K$ exists if and only if $\tau(c)d^{-1~}$ is a square in $K$. If we assume that $D(K,\sigma,c)$ and $D(K,\sigma,d)$ are division algebras, $c,d$ are certainly not squares in $K$ and so must lie in non-identity cosets of $K^{\times}/(K^{\times})^2$. It is clear that $\tau(c)$ must lie in the same coset as $c$.\\
Considering the images of $\tau(c)$ and $d^{-1}$ in the quotient group $K^{\times}/(K^{\times})^2$, it follows that $\tau(c)d^{-1}$ is a square in $K^{\times}$ if and only if $c$ and $d$ lie in the same coset of $K^{\times}/(K^{\times})^2$. As there are 3 non-trivial cosets, we conclude there are at most $3\left| Aut_{\mathbb{Q}_p}(K)\right|$ non-isomorphic commutative Dickson division algebras.
\end{proof}

We cannot say for certain that we attain this bound, as this would assume that there exists a suitable $c\in K^{\times}$ in each non-trivial coset of $K^{\times}/(K^{\times})^2$ such that $D(K,\sigma,c)$ is a division algebra for each $\sigma\in Aut_{\mathbb{Q}_p}(K)$. However, if we can find some $c\in K^{\times}$ that satisfies the conditions of Corollary \ref{CDSDivisionNorm} from each coset of $K^{\times}/(K^{\times})^2$, this is sufficient to show that there are exactly $3\left| Aut_{\mathbb{Q}_p}(K)\right|$ non-isomorphic commutative Dickson division algebras.
%

For an arbitrary field $F$, we conclude the following analogously:

\begin{corollary} Suppose $K/F$ is a finite field extension such that $Aut_F(K)$ is an abelian group and there exists $c\in K^{\times}$ such that $N_{K/F}(c)\neq N_{K/F}(a^2)$ for all $a\in K$. Then there are at least $\left| Aut_F(K)\right|$ non-isomorphic commutative Dickson division algebras over $F$ of the form $D(K,\sigma,c)$.
\end{corollary}
\subsection{Automorphisms}

The automorphisms of commutative Dickson algebras were computed in \cite{Bur} when $K$ is a finite cyclic field extension. We consider the subset $$J(c)=\lbrace \tau\in Aut_F(K)\mid X^2-\tau(c)c^{-1} \mbox{ has solutions in }K\rbrace \subset Aut_F(K),$$ introduced in \cite{Bur}.

\begin{lemma} $J(c)$ is a subgroup of $Aut_F(K)$.\label{JcSubgroup}
\end{lemma}
\begin{proof}
Clearly the identity automorphism is contained in $J(c)$, as $X^2-cc^{-1}=X^2-1$ always has the solutions $X=\pm 1$.\\
Let $\tau,\phi\in J(c)$. Then $\tau(c)c^{-1}=a^2$ and $\phi(c)c^{-1}=b^2$ for some $a,b\in K^{\times}$. It follows that $$\phi\circ\tau(c)c^{-1}=\phi(a^2c)c^{-1}=\phi(a^2)b^2cc^{-1},$$ so $X^2-\phi\circ\tau(c)c^{-1}$ has the solutions $X=\pm \phi(a)b$. This implies $\phi\circ\tau\in J(c)$. Finally, for each $\tau\in J(c)$ we have $\tau^{-1}(c)c^{-1}=\tau^{-1}(a^{-1})^2,$ so $\tau^{-1}\in J(c)$. 
\end{proof}
When $K$ is a cyclic extension, there exists $2\left|J(c)\right|$ automorphisms of $D(K,\sigma,c)$:

\begin{theorem}\label{CDSSemifield2nAuts}\begin{enumerate}[(i)]
\item[(i)] [\cite{Bur}, Theorem 3 in our notation] Let $K$ be a cyclic extension of $F$. Then there exists $2\left|J(c)\right|$ automorphisms of $D(K,\sigma,c)$, each of which is given by $$G(x,y)=(\tau(x),\tau(y)b_i)$$ for each $\tau\in J(c)$, where $b_i\in K$ are such that $\sigma(b_i)$ are the two solutions of $X^2-\tau(c)c^{-1}$ for $i=1,2$.

\item[(ii)] [\cite{Bur}, Theorem 3' in our notation] Let $F$ be a finite field of odd characteristic and $K$ be a finite extension of degree $n$. Then there exists $2n$ automorphisms of $D=D(K,\sigma,c)$, each of which is given by $$G(x,y)=(\tau(x),\tau(y)b_i)$$ for each $\tau\in Aut_F(K)$, where $b_i\in K$ are such that $\sigma(b_i)$ are the two solutions of $X^2-\tau(c)c^{-1}$ for $i=1,2$.
\end{enumerate}
\end{theorem}

We now compute the automorphisms when $K$ is an arbitrary finite field extension. We continue to assume that $\sigma\neq id$. 

\begin{theorem}\label{CDSAutGroup}
All automorphisms $G:D(K,\sigma,c)\to D(K,\sigma,c)$ are of the form $$G(u,v)=(\tau(u),\tau(v)b)$$ for some $\tau\in Aut_F(K)$ such that $\tau$ and $\sigma$ commute and $b\in K^{\times}$ satisfying $\tau(c)=c\sigma(b^2)$. Further, all maps of this form with $\tau\in Aut_F(K)$ and $b\in K^{\times}$ satisfying these conditions yield an automorphism of $D$.
\end{theorem}
\begin{proof}
Let $D=D(K,\sigma,c)$. Suppose that $G\in Aut_F(D)$. As automorphisms preserve the nuclei of an algebra, $G$ restricted to $K$ must be an automorphism of $K$. As $G$ is $F$-linear we obtain $F\subset Fix(G\!\mid_K)$ and so in fact $G\!\mid_K\in Aut_F(K)$. Let $G\!\mid_K=\tau\in Aut_F(K)$, so we have $G(x,0)=(\tau(x),0)$ for all $x\in K$.\\
Let $G(0,1)=(a,b)$ for some $a,b\in K$. Then we have \begin{equation*}
G(x,y)=G(x,0)+G(0,1)G(y,0)+(\tau(x)+a\tau(y),\tau(y)b).
\end{equation*}
As $G$ is multiplicative, we must also have $G((0,1)^2)=G(0,1)^2$ which holds if and only if $$(a,b)(a,b)=(\tau(c),0).$$ From this, we obtain the equations $a^2+c\sigma(b^2)=\tau(c)$ and $2ab=0.$
As $K$ does not have characteristic 2, this implies that either $a=0$ or $b=0$. If $b=0$, then $G(x,y)=(\tau(x)+\tau(y)a,0)$ and so $G$ is not surjective. This is a contradiction, as $G$ is an automorphism. Thus $a=0$ and we obtain $c\sigma(b^2)=\tau(c)$.\\
Finally, as $G$ is multiplicative we have $G(u,v)G(x,y)=G((u,v)(x,y))$ for all $u,v,x,y\in K$. Computing both sides of this equation, we get $$(\tau(ux)+c\sigma(\tau(vy)b^2),\tau(uy)b+\tau(vx)b)=(\tau(ux+c\sigma(vy)),\tau(uy+vx)b)$$ for all $u,v,x,y\in K$, which implies that $c\sigma(\tau(vy)b^2)=\tau(c\sigma(vy))$. After substituting the condition $c\sigma(b^2)=\tau(c)$, we are left with $\sigma(\tau(vy))=\tau(\sigma(vy))$ for all $v,y\in K$; that is, $\tau$ and $\sigma$ must commute.\\
Conversely, let $G:D\to D$ be a map defined by $G(x,y)=(\tau(x),\tau(y)b)$ such that $\tau$ and $\sigma$ commute and $\tau(c)=c\sigma(b^2)$. It is easily checked that this map is $F$-linear, bijective, additive and multiplicative. Hence $G$ is an $F$-algebra automorphism of $D$.
\end{proof}

\begin{corollary}\label{CDSAutSubgroupb1} There is a subgroup of $Aut_F(D)$ isomorphic to $$\lbrace\tau\in Aut_F(K)\mid \tau(c)=c \mbox{ and } \tau\circ\sigma=\sigma\circ\tau\rbrace.$$
\end{corollary}
\begin{proof}
By Theorem \ref{CDSAutGroup}, all automorphisms of $D$ are of the form $G(x,y)=(\tau(x),\tau(y)b)$, such that $\tau$ and $\sigma$ commute and $b\in K^{\times}$ satisfies $\tau(c)=c\sigma(b^2)$. If we let $b=1$, we obtain a subgroup of $Aut_F(D)$ such that $\tau$ and $\sigma$ commute and $\tau(c)=c$.
\end{proof}

The subset of $Aut_F(K)$ containing all the automorphisms of $K$ which commute with $\sigma\in Aut_F(K)$ is called the \textit{centralizer of $\sigma$ in $Aut_F(K)$} and is denoted by $$C(\sigma)=\lbrace \tau\in Aut_F(K) \mid \tau\circ\sigma=\sigma\circ\tau\rbrace.$$ This subset forms a subgroup of $Aut_F(K)$, so $J(c)\cap C(\sigma)$ is also a subgroup of $Aut_F(K)$. We get the following generalisation of \cite[Theorem 3]{Bur}:

\begin{theorem}\label{CDSExactNumberOfAutomorphisms}
There are exactly $2\left|J(c)\cap C(\sigma)\right|$ automorphisms of $D(K,\sigma,c)$, each of which is given by $$G(x,y)=(\tau(x),\tau(y)b_i)$$ for each $\tau\in J(c)\cap C(\sigma)$, where $b_i\in K^{\times}$ is chosen such that $\sigma(b_i)$ are the two solutions of $X^2-\tau(c)c^{-1}$ for $i=1,2$.
\end{theorem}
\begin{proof}
By Theorem \ref{CDSAutGroup}, $G$ is an automorphism of $D(K,\sigma,c)$ if and only if $G(u,v)=(\tau(u),\tau(v)b)$ for some $\tau\in C(\sigma)$ and $b\in K^{\times}$ such that $\sigma(b)^2=\tau(c)c^{-1}$. We can find such $b\in K^{\times}$ if and only if $\tau\in J(c)$. Denote the solutions of $X^2-\tau(c)c^{-1}$ by $\sigma(b_1)$ and $\sigma(b_2)$. Thus $G$ is an automorphism of $D(K,\sigma,c)$ if and only if $G(u,v)=(\tau(u),\tau(v)b_i)$ for each $\tau\in J(c)\cap C(\sigma)$, where $b_i\in K$ are such that $\sigma(b_i)$ are the two solutions of $X^2-\tau(c)c^{-1}$ for $i=1,2$.
\end{proof}

\begin{corollary}\label{CDSExactNumberOfAutomorphismsAbelian} If $Aut_F(K)$ is abelian, then $D(K,\sigma,c)$ has exactly $2\left|J(c)\right|$ automorphisms.
\end{corollary}
\begin{proof}
This follows immediately from Theorem \ref{CDSExactNumberOfAutomorphisms} after noting that $C(\sigma)=Aut_F(K)$.
\end{proof}

\begin{corollary}\label{CDSCSigmaAuts} If $c\in F^{\times}$, then $D(K,\sigma,c)$ has exactly $2\left|C(\sigma)\right|$ automorphisms.
\end{corollary}
\begin{proof}
As $c\in F^{\times}$, for all $\tau\in Aut_F(K)$ we have $$X^2-\tau(c)c^{-1}=X^2-cc^{-1}=X^2-1,$$ which always has the solutions $X=\pm 1$. This yields $J(c)=Aut_F(K)$. The result then follows from Theorem \ref{CDSExactNumberOfAutomorphisms}.
\end{proof}

As $J(c)\cap C(\sigma)$ forms a subgroup of $Aut_F(K)$, we know that $\left|J(c)\cap C(\sigma)\right|$ must divide $\left|Aut_F(K)\right|$. Due to this, we can easily determine the exact size of the automorphism group of $D(K,\sigma,c)$ in certain cases.

\begin{corollary} If $K$ is a field extension of prime degree $p$ over $F$, $J(c)$ is equal to either $\lbrace id\rbrace$ or $Aut_F(K)$. Further, $\left|Aut_F(D(K,\sigma,c))\right|\in\lbrace 2,2p\rbrace$.
\end{corollary}
\begin{proof}
Let $[K:F]=p$ for some prime $p$. Then $Aut_F(K)$ is necessarily cyclic and hence abelian. As $\left|Aut_F(K)\right|=p$, we must have $\left|J(c)\right|\in \lbrace 1,p\rbrace$ and so $J(c)=\lbrace id\rbrace$ or $J(c)=Aut_F(K)$. The remainder of the result follows from Corollary \ref{CDSExactNumberOfAutomorphismsAbelian}.
 \end{proof}
 
\begin{corollary}\label{QpJc=Aut(K)} If $F=\mathbb{Q}_p$ for $p\neq 2$, then $J(c)=Aut_{\mathbb{Q}_p}(K)$ and $\left|Aut_{\mathbb{Q}_p}(D(K,\sigma,c))\right|=2\left| C(\sigma)\right|$.
\end{corollary}
\begin{proof}
As $\tau(c)$ and $c^{-1}$ clearly lie in the same coset of $K^{\times}/(K^{\times})^2$, it follows that $\tau(c)c^{-1}\in K^2$ for all $\tau\in Aut_{\mathbb{Q}_p}(K)$. We conclude that $J(c)=Aut_{\mathbb{Q}_p}(K)$ and thus $\left|Aut_{\mathbb{Q}_p}(D(K,\sigma,c))\right|=2\left| C(\sigma)\right|$ by Theorem \ref{CDSExactNumberOfAutomorphisms}.
\end{proof}
 
Generally it is difficult to actually calculate $J(c)$, so we instead bound the size of $Aut_F(D(K,\sigma,c))$. We already have an upper bound as a consequence of Theorem \ref{CDSAutGroup}.
All the elements of $Aut_F(K)$ which act as the identity on $c$ form a subgroup of $Aut_F(K)$ called the \textit{isotropy group of $c$}, denoted by $$Aut_F(K)_c=\lbrace \tau\in Aut_F(K)\mid\tau(c)=c\rbrace.$$ By Corollary \ref{CDSAutSubgroupb1}, there is a subgroup of $Aut_F(D(K,\sigma,c))$ which is isomorphic to $C(\sigma)\cap Aut_F(K)_c.$ This allows us to bound the size of the automorphism group of $D(K,\sigma,c)$ from below:

\begin{theorem} There are between $2\left|C(\sigma)\cap Aut_F(K)_c\right|$ and $2\left|C(\sigma)\right|$ automorphisms of $D(K,\sigma, c)$.
\end{theorem}
\begin{proof}
It is clear that $J(c)\cap C(\sigma)$ is a subgroup of $C(\sigma)$. Each $\tau\in C(\sigma)$ can be used to construct at most 2 automorphisms of $D(K,\sigma,c)$ corresponding to the two possible solutions of $X^2-\tau(c)c^{-1}$, so we have $\left|Aut_F(D(K,\sigma,c))\right|\leqslant 2\left|C(\sigma)\right|.$\\
Additionally, each $\tau\in C(\sigma)\cap Aut_F(K)_c$ can be used to construct the maps $(x,y)\mapsto (\tau(x),\pm\tau(y)).$ It follows from Theorem \ref{CDSAutGroup} that these are automorphisms of $D(K,\sigma,c)$, so $2\left|C(\sigma)\cap Aut_F(K)_c\right|\leqslant \left|Aut_F(D(K,\sigma,c))\right|$.
\end{proof}

Wene \cite{Wene} derived an alternative description of the automorphism group of $D(K,\sigma,c)$ when $K$ is a finite field, in terms of inner automorphisms. An automorphism $\theta$ of $D(K,\sigma,c)$ is an \textit{inner automorphism} if there exists $m\in D(K,\sigma,c)$ with left inverse $m_l^{-1}$ such that $$\theta(x)=(m_l^{-1}x)m$$ for all $x\in D(K,\sigma,c)$. The proof given in \cite[Theorem 18]{Wene} holds verbatim for any finite field extension, yielding a sufficient condition for the existence of (nontrivial) inner automorphisms of a commutative Dickson algebra:

\begin{theorem}[\cite{Wene}, Theorem 18]
Let $D(K,\sigma,c)$ be a division algebra. Denote $\lambda=(0,1)$. Then $$\Phi(x,y)=[\lambda_l^{-1}(x,y)]\lambda=(\sigma(x),\sigma(y))$$ defines an inner automorphism of $D(K,\sigma,c)$ if and only if $\sigma(c)=c$.
\end{theorem}

\subsection{The group structure of $Aut_F(D)$}

By Theorem \ref{CDSAutGroup}, we know that all the $2\left|C(\sigma)\cap J(c)\right|$ automorphisms of $D$ are of the form $G(u,v)=(\tau(u),\tau(v)b)$ for some $\tau\in C(\sigma)\cap J(c)$ and $b\in K^{\times}$ such that $\sigma(b)^2=\tau(c)c^{-1}$. Note that this final condition is equivalent to $b\in K^{\times}$ being a solution of $X^2-\sigma^{-1}(\tau(c)c^{-1})\in K[X].$ We will denote the solutions of this polynomial by $b_{\tau,1}$ and $b_{\tau,2}$. As the characteristic of $F$ is not 2, it is clear that $b_{\tau,2}=-b_{\tau,1}$.

\begin{lemma}\label{StructureofAutGroupPowers} Let $b_{\tau,1},b_{\tau,2}$ be the two solutions of $X^2-\sigma^{-1}(\tau(c)c^{-1})$ and suppose $\tau^n=id$. Then $b_{\tau,i}\tau(b_{\tau,i})\tau^2(b_{\tau,i})...\tau^{n-1}(b_{\tau,i})=\pm 1.$\\ Moreover, if $n$ is odd, we have $b_{\tau,i}\tau(b_{\tau,i})\tau^2(b_{\tau,i})...\tau^{n-1}(b_{\tau,i})=1$ for $i=1$ or $i=2$, but not both.
\end{lemma}
\begin{proof}
As in the proof of Lemma \ref{JcSubgroup}, if $b_{\tau}$ and $b_{\phi}$ are solutions of $X^2-\sigma^{-1}(\tau(c)c^{-1})$ and $X^2-\sigma^{-1}(\phi(c)c^{-1})$ respectively then the equation $$X^2-\sigma^{-1}(\phi\circ\tau(c)c^{-1})$$ has the solutions $X=\pm\phi(b_{\tau})b_{\phi}$. Similarly the equation $X^2-\sigma^{-1}(\tau^2(c)c^{-1})$ has the solutions $X=\pm\tau(b_{\tau})b_{\tau}$, the equation $X^2-\sigma^{-1}(\tau^3(c)c^{-1})$ has the solutions $X=\pm\tau(b_{\tau^2})b_{\tau}=\tau^2(b_\tau)\tau(b_{\tau})b_\tau$, and so on. Hence we see that for $i=1,2$ $$b_{\tau,i}\tau(b_{\tau,i})\tau^2(b_{\tau,i})...\tau^{n-1}(b_{\tau,i})$$ is a solution of $X^2-\sigma^{-1}(\tau^n(c)c^{-1})$. As $\tau^n=id$, we also conclude that the solutions of $$X^2-\sigma^{-1}(\tau^n(c)c^{-1})=X^2-\sigma^{-1}(cc^{-1})=X^2-1$$ are $X=\pm 1$ and so $b_{\tau,i}\tau(b_{\tau,i})\tau^2(b_{\tau,i})...\tau^{n-1}(b_{\tau,i})=\pm 1.$ As $b_{\tau,2}=-b_{\tau,1}$, we have $$b_{\tau,2}\tau(b_{\tau,2})\tau^2(b_{\tau,2})...\tau^{n-1}(b_{\tau,2})\\
= (-1)^nb_{\tau,1}\tau(b_{\tau,1})\tau^2(b_{\tau,1})...\tau^{n-1}(b_{\tau,1}).$$
If $n$ is odd, this implies that $b_{\tau,2}\tau(b_{\tau,2})\tau^2(b_{\tau,2})...\tau^{n-1}(b_{\tau,2})= -b_{\tau,1}\tau(b_{\tau,1})\tau^2(b_{\tau,1})...\tau^{n-1}(b_{\tau,1})$ and the result follows.
\end{proof}

\begin{theorem}
For all $D(K,\sigma,c)$, we have $$Aut_F(D(K,\sigma,c))\cong (C(\sigma)\cap J(c))\times \mathbb{F}_2.$$
\end{theorem}
\begin{proof}
As $C(\sigma)\cap J(c)$ is a finite group, there exists a minimal generating set $\lbrace \tau_1,...,\tau_m\rbrace.$ Let $\tau$ be an element of this generating set and let $b_{\tau,i}$ ($i=1,2$) be the two roots of $X^2-\sigma^{-1}(\tau(c)c^{-1})$. As $J(c)$ is a finite group, $\tau^n$ must be equal to the identity for some $n>1$. By Lemma \ref{StructureofAutGroupPowers}, this implies $$b_{\tau,i}\tau(b_{\tau,i})\tau^2(b_{\tau,i})...\tau^{n-1}(b_{\tau,i})=\pm 1$$ for $i=1,2$. If $n$ is odd, relabel the roots such that $b_{\tau,1}$ satisfies $$b_{\tau,1}\tau(b_{\tau,1})\tau^2(b_{\tau,1})...\tau^{n-1}(b_{\tau,1})=1$$ and $b_{\tau,2}$ satisfies $$b_{\tau,2}\tau(b_{\tau,2})\tau^2(b_{\tau,2})...\tau^{n-1}(b_{\tau,2})=-1.$$ Henceforth, we will denote $b_{\tau,1}=b_\tau$. Now let $\phi\in C(\sigma)\cap J(c)$. As $\lbrace \tau_1,...,\tau_m\rbrace$ generates $C(\sigma)\cap J(c)$, $\phi$ can be expressed as a product of the $\tau_i$. Due to this, we can construct the roots of $X^2-\sigma^{-1}(\phi(c)c^{-1})$ from the $b_{\tau_i}$. For example, if $\phi=\tau_i\circ\tau_j$ then we obtain $$b_{\phi}=b_{\tau_i}\tau_i(b_{\tau_j}).$$ This method can be applied recursively to construct the roots of $X^2-\sigma^{-1}(\tau(c)c^{-1})$ for all $\tau\in C(\sigma)\cap J(c)$.\\
We can now express all automorphisms of $D$ in the form $G(u,v)=(\tau(u),\pm\tau(v)b_{\tau})$ for some $\tau\in J(c)\cap C(\sigma)$ and $b_{\tau}$ as defined above. Define a map $\Phi: Aut_F(D)\to (J(c)\cap C(\sigma))\times \mathbb{F}_2$ by $$\Phi(G)=(\tau,\pm 1).$$ This map is well-defined due to the careful labelling of roots of $X^2-\sigma^{-1}(\tau(c)c^{-1})$. It is easy to see that it gives an isomorphism between groups.
\end{proof}

\begin{corollary} If $F=\mathbb{Q}_p$, then $Aut_{\mathbb{Q}_p}(D(K,\sigma,c))\cong C(\sigma)\times \mathbb{F}_2.$
\end{corollary}
\begin{proof}
This follows from Corollary \ref{QpJc=Aut(K)}.
\end{proof}

Thus it is sufficient to consider the subgroups of $Aut_F(K)$, $C(\sigma)$ and $J(c)$, in order to determine the structure of the automorphism groups of these algebras.

\section{A generalisation of the commutative Dickson algebras construction obtained by doubling a central simple algebra}\label{FurtherGeneralisingDicksonAlgebras}

Let $B$ be a central simple algebra over $F$. Let $\sigma\in Aut_F(B)$ be a non-trival automorphism and $c\in B^{\times}$. As $B$ is not commutative, we can generalise the classical Dickson multiplication on the $F$-vector space $B\oplus B$ in three ways: \begin{itemize}
\item $(u,v)\circ(x,y)=(ux+c\sigma(vy),uy+vx),$
\item $(u,v)\circ(x,y)=(ux+\sigma(v)c\sigma(y),uy+vx),$
\item $(u,v)\circ(x,y)=(ux+\sigma(vy)c,uy+vx).$
\end{itemize}
We denote the $F$-vector space $B\oplus B$ endowed with each of these multiplications by $D(B,\sigma,c)$, $D_m(B,\sigma,c)$ and $D_r(B,\sigma,c)$, respectively. If $c\in F^{\times}$, the three constructions are identical. All three constructions yield unital nonassociative algebras over $F$ and are canonical generalisations of the commutative construction defined by Dickson.

\begin{lemma}\label{CommD}\begin{enumerate}[(i)]
\item Let $D=D(B,\sigma,c)$ or $D=D_r(B,\sigma,c)$. Then $Comm(D)=F\oplus F$.
\item Let $D=D_m(B,\sigma,c)$. If $c\in F^{\times}$, then $Comm(D)=F\oplus F$. Otherwise, $Comm(D)=F$.
\end{enumerate}
\end{lemma}
\begin{proof}
\begin{enumerate}[(i)]
\item We only show the proof for $D(B,\sigma,c)$ as the proof for $D_r(B,\sigma,c)$ follows identically. Let $(u,v)\in Comm(D)$. Then for all $x\in B$, we have $$(u,v)(x,0)=(x,0)(u,v).$$ This is equivalent to $ux=xu$ and $vx=xv.$ This holds for all $x\in B$ if and only if both $u$ and $v$ lie in the centre of $B$. Hence $Comm(D)\subseteq F\oplus F$. It is easily checked that all elements of $F\oplus F$ are contained in $Comm(D)$. Hence $Comm(D)=F\oplus F$.
\item Let $(u,v)\in Comm(D)$. Then for all $x\in B$, we have $(u,v)(0,x)=(0,x)(u,v).$ This is equivalent to $\sigma(v)c\sigma(x)=\sigma(x)c\sigma(v)$ and $ux=xu.$
The second equation implies that $u\in Z(B)=F$. If $c\not\in F$, then the first equation is only satisfied for all $x\in B$ when $v=0$, which yields $Comm(D)=F$.\\
If $c\in F^{\times}$, we have $D_m(B,\sigma,c)=D(B,\sigma,c)$ and so by (i), we obtain that $Comm(D)=F\oplus F.$
\end{enumerate}
\end{proof}

\begin{theorem}\label{CSANucleus}
Let $D=D(B,\sigma,c)$. Then \begin{itemize}
\item $Nuc_l(D)=\lbrace k\in B \mid c\sigma(k)=kc\rbrace\subset B$,
\item $Nuc_m(D)=B$,
\item $Nuc_r(D)=Fix(\sigma).$
\end{itemize}
In particular, $Nuc(D)=Fix(\sigma)\cap\lbrace k\in B\mid c\sigma(k)=kc\rbrace$ and $Z(D)=F$.
\end{theorem}
\begin{proof}
We will show the proof for the left nucleus. The calculations for the middle and right nucleus are obtained similarly. \\
Suppose $(k,l)$ lies in the left nucleus for some $k,l\in B$. Then for all $x\in B$, we must have $$((k,l)(0,1))(x,0)=(k,l)((0,1)(x,0)).$$ Computing both sides of this it follows that $$(c\sigma(l)x,kx)=(c\sigma(lx),kx).$$ As $\sigma$ is a non-trivial automorphism of $B$, this is true for all $x\in B$ if and only if $l=0$.
Thus we only need to consider elements of the form $(k,0)$ for $k\in B$. Now $(k,0)\in Nuc_l(D)$ if and only if we obtain $$((k,0)(u,v))(x,y)=(k,0)((u,v)(x,y))$$ for all $u,v,x,y\in B$. Computing both sides of this, this yields $$(kux+c\sigma(kvy),kuy+kvx)=(kux+kc\sigma(vy),kuy+kvx).$$ This is satisfied for all $u,v,x,y\in B$ if and only if $c\sigma(k)=kc$. Hence we have that $$Nuc_l(D)=\lbrace (k,0)\mid k\in B \mbox{ such that }c\sigma(k)=kc\rbrace.$$
As the centre is the intersection of the nucleus and the commutator, this yields $Z(D)=(Fix(\sigma)\cap\lbrace k\in B\mid c\sigma(k)=kc\rbrace\cap F)\oplus 0=F\oplus 0$.
\end{proof}

Similarly, we can calculate the left, middle and right nuclei and centre of $D_r(B,\sigma,c)$ and $D_m(B,\sigma,c)$:

\begin{theorem}\label{CSANucleusR}
Let $D=D_r(B,\sigma,c)$. Then \begin{itemize}
\item $Nuc_l(D)=Fix(\sigma)$,
\item $Nuc_m(D)=B$,
\item $Nuc_r(D)=\lbrace k\in B \mid c\sigma(k)=kc\rbrace\subset B.$
\end{itemize}
In particular, $Nuc(D)=Fix(\sigma)\cap\lbrace k\in B\mid c\sigma(k)=kc\rbrace$ and $Z(D)=F$.
\end{theorem}

\begin{theorem}\label{CSANucleusM}
Let $D=D_m(B,\sigma,c)$. Then \begin{itemize}
\item $Nuc_l(D)=Fix(\sigma)$,
\item $Nuc_m(D)=\lbrace k\in B\mid \sigma(k)c=c\sigma(k)\rbrace \subset B$,
\item $Nuc_r(D)=Fix(\sigma).$
\end{itemize}
In particular, $Nuc(D)=Fix(\sigma)\cap\lbrace k\in B\mid c\sigma(k)=\sigma(k)c\rbrace$ and $Z(D)=F$.
\end{theorem}

Note that if $c\in F^{\times}$, the three algebras we obtain are identical as noted earlier. In this case, the left and right nuclei are equal to $Fix(\sigma)$ and the middle nucleus is equal to $B$.
%

\begin{corollary}
Let $c\in B\setminus F^{\times}$. Then \begin{itemize}
\item $D(B,\sigma,c)\not\cong D_m(B,\sigma,c)$,
\item $D_m(B,\sigma,c)\not\cong D_r(B,\sigma,c)$.
\end{itemize}
If $c$ does not commute with all elements of $Fix(\sigma)$, then $D(B,\sigma,c)\not\cong D_r(B,\sigma,c)$.
\end{corollary}
\begin{proof}
Since automorphisms preserve each of the left, middle and right nuclei, if $D(B,\sigma,c)\cong D_m(B,\sigma,c)$ this implies that $\lbrace k\in B\mid \sigma(k)c=c\sigma(k)\rbrace=B.$ As $c\not\in F$, we can find $k\in B$ such that $\sigma(k)$ does not commute with $c$ so this is never true. An identical argument shows that $D_m(B,\sigma,c)\not\cong D_r(B,\sigma,c)$.\\
Finally, we see that $D(B,\sigma,c)\cong D_r(B,\sigma,c)$ occurs only if $Fix(\sigma)=\lbrace k\in B\mid kc=c\sigma(k)\rbrace.$ Let $x\in Fix(\sigma)$. We have $x\in \lbrace k\in B\mid kc=c\sigma(k)\rbrace$ if and only if $cx=xc.$\\ Similarly, if we take an element $y\in \lbrace k\in B\mid kc=c\sigma(k)\rbrace$, it lies in $Fix(\sigma)$ if and only if $cy=yc$. Thus the left nuclei of the two algebras are equal only when $c$ commutes with all of $Fix(\sigma)$. Otherwise, we must have $D(B,\sigma,c)\not\cong D_r(B,\sigma,c)$.
\end{proof}

Similarly to the algebras we obtained from doubling a field extension, any $F$-subalgebra of $B$ appears as a subalgebra of $D(B,\sigma,c)$, $D_m(B,\sigma,c)$ and $D_r(B,\sigma,c)$. Additionally, if $E\subset B$ is such that $c\in E^{\times}$ and $\sigma\!\mid_E\in Aut_F(E)$, then $D(E,\sigma\!\mid_E,c)$ (resp. $D_m(E,\sigma\!\mid_E,c)$ and $D_r(E,\sigma\!\mid_E,c)$) is a subalgebra of $D(B,\sigma,c)$ (resp. $D_m(B,\sigma,c)$ and $D_r(B,\sigma,c)$). In particular, this yields the following:

\begin{theorem}
If $c\in K^{\times}$ for some separable field extension $K/F$ contained in $B$ such that $\sigma\!\mid_K=\phi\in Aut_F(K)$, then $D(K,\phi,c)$ is a commutative Dickson subalgebra of $D(B,\sigma,c)$, $D_m(B,\sigma,c)$ and $D_r(B,\sigma,c)$.
\end{theorem}

\begin{theorem}\label{CDSDivisionB}
\begin{enumerate}[(i)]
\item $D=D(B,\sigma,c)$ is a division algebra if and only if $c\neq rt^{-1}rs\sigma(s^{-1}t^{-1})$ for all $r,s,t\in B^{\times}$.
\item $D_m(B,\sigma,c)$ is a division algebra if and only if $c\neq \sigma(t)^{-1}rt^{-1}rs\sigma(s)^{-1}$ for all $r,s,t\in B^{\times}$. 
\item $D_r(B,\sigma,c)$ is a division algebra if and only if $c\neq \sigma(s^{-1}t^{-1})rt^{-1}rs$ for all $r,s,t\in B^{\times}$.
\end{enumerate}
\end{theorem}
\begin{proof}
(i): Suppose that $D$ is not a division algebra. Then there exist nonzero elements $(u,v),(x,y)\in K\oplus K$ such that $(u,v)(x,y)=(0,0).$ This is equivalent to the simultaneous equations \begin{align}
ux+c\sigma(vy)=&0, \label{CDSDivisionB1}\\
uy+vx=&0. \label{CDSDivisionB2}
\end{align}
If $v=0$, then (\ref{CDSDivisionB2}) becomes $uy=0$, so either $u=0$ or $y=0$. However, $u$ must be nonzero, else $(u,v)=(0,0)$ which is a contradiction, so we must have $y=0$. Additionally, (\ref{CDSDivisionB1}) gives $ux=0$. As $u$ is nonzero, this implies $x=0$ and so $(x,y)=(0,0)$ which is again a contradiction.\\
So let $v\neq 0$. As $B$ is an associative division algebra, we have $v^{-1}\in B$ and hence we obtain $$x=-v^{-1}uy$$ from (\ref{CDSDivisionB2}). Now if $y=0$, this implies that $x=0$ which is a contradiction to $(x,y)\neq(0,0)$. Substituting this into (\ref{CDSDivisionB1}), we get $$-uv^{-1}uy+c\sigma(vy)=0,$$ which rearranges to give $c=uv^{-1}uy\sigma(y)^{-1}\sigma(v)^{-1}$.\\
Conversely, suppose $c=rt^{-1}rs\sigma(s)^{-1}\sigma(t)^{-1}$ for some $r,s,t\in K^{\times}$. Consider the elements $(r,t)$ and $(-t^{-1}rs,s)$. Both elements are nonzero but satisfy \begin{align*}
(r,t)(-t^{-1}rs,s)=&(-rt^{-1}rs+rt^{-1}rs\sigma(s)^{-1}\sigma(t)^{-1}\sigma(ts),rs-tt^{-1}rs)\\
=&(0,0).
\end{align*}
Hence $D$ is not a division algebra.\\
The proofs of (ii) and (iii) follow almost identically to (i).
\end{proof}

\begin{corollary} If $c\in B^{\times 2}$, then $D(B,\sigma,c), D_m(B,\sigma,c)$, and $D_r(B,\sigma,c)$ are not division algebras.
\end{corollary}
\begin{proof}
This follows from setting $s=t=1$ in Theorem \ref{CDSDivisionB}.
\end{proof}

\begin{corollary} Let $N_{B/F}:B\to F$ be the nondegenerate multiplicative norm form on $B$. The algebras $D=D(B,\sigma,c)$, $D_m(B,\sigma,c)$, $D_r(B,\sigma,c)$ are division algebras if $$N_{B/F}(c)\neq N_{B/F}(a)^2$$ for all $a\in B$.
\end{corollary}
\begin{proof}
This follows analogously to Corollary \ref{CDSDivisionNorm}.
\end{proof}

\begin{example}
\begin{enumerate}[(i)]
\item Let $F=\mathbb{Q}$ and $B=(a,b)$ be a quaternion division algebra over $\mathbb{Q}$ with $a,b>0$. For all $x\in B^{\times}$, we see that $N_{B/\mathbb{Q}}(x)^2> 0$; as a consequence, $D(B,\sigma,c)$ is a division algebra for any $c\in B^{\times}$ such that $N_{B/\mathbb{Q}}(c)<0$. For example, if we pick $c=c_1i+c_2j$ for some $c_i\in \mathbb{Q}$ not both zero, then $$N_{B/\mathbb{Q}}(c)=-c_1^2a-c_2^2b<0,$$ so $D(B,\sigma,c)$ is a division algebra.
\item Let $F=\mathbb{Q}_p$ and $B=(u,p)$ be the unique quaternion division algebra over $\mathbb{Q}_p$ for some $u\in \mathbb{Z}_p\setminus (\mathbb{Z}_p)^2$. Then for all $c\in B$, it follows that $N_{B/\mathbb{Q}_p}(c)=x^2-y^2u-z^2p+w^2up$ for some $x,y,z,w\in \mathbb{Q}_p$. As $up$ is not a square in $\mathbb{Q}_p$, for any $c\in B$ such that $N_{B/\mathbb{Q}_p}(c)=w^2up$ we conclude that $D(B,\sigma,c)$ is a division algebra over $\mathbb{Q}_p$.
\end{enumerate}
\end{example}

\subsection{Isomorphisms}

The results and proofs from Section 2 regarding isomorphisms and automorphisms of commutative Dickson algebras generalise almost identically to $D(B,\sigma,c)$ and $D_r(B,\sigma,c)$, as the middle nuclei of these algebras are equal to $B$. First note the following result:

\begin{lemma}\label{otherdicksonnucleus} Let $D=D(B,\sigma,c)$, $D'=D(B,\phi,d)$ be two Dickson algebras over $F$. If there exists an $F$-isomorphism $\tau:B\to B'$ such that $\tau\circ\sigma=\phi\circ\tau$ and $\tau(c)=db^2$ for some $b\in F^{\times}$, then $\tau\!\mid_{Nuc_l(D)}:Nuc_l(D)\to Nuc_l(D')$ is an $F$-isomorphism.
\end{lemma}
\begin{proof}
As with the proof of Lemma \ref{sigmataufixfields}, we only need to show that $im(\tau\!\mid_{Nuc_l(D)})=Nuc_l(D')$. First, consider $x\in Nuc_l(D)$. It follows that $x$ must satisfy $c\sigma(k)=kc$. Applying $\tau$ to both sides of the equation and substituting in the condition on $\tau(c)$, we obtain $$db^2\tau(\sigma(k))=\tau(k)db^2.$$ As $b\in F^{\times}$, we can cancel this from both sides. After substituting $\tau\circ\sigma=\phi\circ\tau$, this yields $d\phi(\tau(k))=\tau(k)d$ and thus $\tau(k)\in Nuc_l(D')$. Hence $im(\tau\!\mid_{Nuc_l(D)})\subseteq Nuc_l(D')$. In order to show equality, we follow an analogous process to the one in the proof of Lemma \ref{sigmataufixfields}.
\end{proof}

It is clear that this also holds when considering the right nucleus of $D_r(B,\sigma,c)$, as this is equal to the left nucleus of $D(B,\sigma,c)$. We will always assume that $B, B'$ are central simple division algebras over $F$. We now give a proof of the generalisation of Theorem \ref{CDSIsomorphisms}:

\begin{theorem}\label{CDSBIsomorphisms}
Let $D=D(B,\sigma,c)$ and $D'=D(B',\phi,d)$ be $F$-algebras. Then $G:D\to D'$ is an isomorphism if and only if $G$ has the form $$G(x,y)=(\tau(x),\tau(y)b)$$ for some $F$-isomorphism $\tau:B\to B'$ such that $\phi\circ\tau=\tau\circ\sigma$ and $\tau(c)=db^2$ for some $b\in F^{\times}$.
\end{theorem}
\begin{proof}
Suppose $G:D\to D'$ is an $F$-isomorphism. Then $G$ maps the middle nucleus of $D$ to the middle nucleus of $D'$, so by Theorem \ref{CSANucleus} this implies $B\cong B'.$ This means $G$ restricted to $B$ must be an isomorphism which maps to $B'$; that is, $G\!\mid_B=\tau:B\to B'$, so this yields $G(x,0)=(\tau(x),0)$ for all $x\in B$.\\
Let $G(0,1)=(a,b)$ for some $a,b\in B'$. Then we have $G(x,y)=G(x,0)+G(0,1)G(y,0)=(\tau(x)+a\tau(y),\tau(y)b),$ and $G(x,y)=G(x,0)+G(y,0)G(0,1)=(\tau(x)+\tau(y)a,b\tau(y)).$ This implies that $a,b\in Z(B')=F$.\\
As $G$ is multiplicative, it follows that $G((0,1)^2)=G(0,1)^2$ which holds if and only if $(a,b)(a,b)=(\tau(c),0).$ From this, we obtain the equations \begin{equation*}
a^2+d\phi(b^2)=\tau(c), \ ab+ba=0.
\end{equation*}
Since we established that $a,b\in F$, this simplifies to $a^2+db^2=\tau(c)$ and $2ab=0.$ As $F$ does not have characteristic 2, this implies that either $a=0$ or $b=0$. If $b=0$, then $G(x,y)=(\tau(x)+\tau(y)a,0)$ and so $G$ is not surjective. This is a contradiction, as $G$ is an isomorphism. Thus $a=0$ and we obtain $db^2=\tau(c)$.\\
Finally, as $G$ is multiplicative it follows that $G(u,v)G(x,y)=G((u,v)(x,y))$ for all $u,v,x,y\in K$. Computing both sides of this equation, we get $$(\tau(ux)+d\phi(\tau(v)b\tau(y)b),\tau(uy)b+\tau(v)b\tau(x))=(\tau(ux+c\sigma(vy)),\tau(uy+vx)b)$$ for all $u,v,x,y\in K$. As $b\in F$, this implies $db^2\phi(\tau(vy))=\tau(c\sigma(vy)).$ After substituting the condition $\tau(c)=d\phi(b^2)$, we conclude $\phi\circ\tau=\tau\circ\sigma.$\\
Conversely, let $G:B\oplus B\to B'\oplus B'$ be defined by $G(x,y)=(\tau(x),\tau(y)b)$ for some $F$-isomorphism $\tau:B\to B'$ such that $\phi\circ\tau=\tau\circ\sigma$ and $\tau(c)=db^2$ for some $b\in F^{\times}$. By Lemma \ref{sigmataufixfields} and Lemma \ref{otherdicksonnucleus}, we see that $G$ maps the nuclei of $D$ isomorphically to the nuclei of $D'$. Thus, it is easily checked that this $G$ gives an $F$-algebra isomorphism from $D$ to $D'$.
\end{proof}

\begin{theorem}\label{CDSBIsomorphismsR}
Let $D=D_r(B,\sigma,c)$ and $D'=D_r(B',\phi,d)$ be $F$-algebras. Then $G:D\to D'$ is an isomorphism if and only if $G$ has the form $$G(x,y)=(\tau(x),\tau(y)b)$$ for some $F$-isomorphism $\tau:B\to B'$ such that $\phi\circ\tau=\tau\circ\sigma$ and $\tau(c)=db^2$ for some $b\in F^{\times}$.
\end{theorem}
\begin{proof}
The proof is analogous to Theorem \ref{CDSBIsomorphisms}, as the middle nuclei of $D_r(B,\sigma,c)$ and $D_r(B',\phi,d)$ are equal to $B$ and $B'$ respectively. Due to this, we can construct the isomorphisms in the same way as in the previous proof.
\end{proof}

\begin{corollary}\label{CDSBIsomorphisms 2}Let $D=D(B,\sigma,c)$ (resp. $D_r(B,\sigma,c)$)  and $D'=D(B,\phi,d)$ (resp. $D_r(B,\phi,d)$) be $F$-algebras. Then $G:D\to D'$ is an isomorphism if and only if $G$ has the form $$G(x,y)=(\tau(x),\tau(y)b)$$ for some $F$-isomorphism $\tau\in Aut_F(B)$ such that $\phi\circ\tau=\tau\circ\sigma$ and $\tau(c)=db^2$ for some $b\in F^{\times}$.
\end{corollary}

\begin{corollary} If $c\in F^{\times}$ and $d\in B^{\times}\setminus F$, then $D(B,\sigma,c)$ is not isomorphic to any of $D(B,\sigma,d)$, $D_m(B,\sigma,d)$ or $D_r(B,\sigma,d)$.
\end{corollary}
\begin{proof}
If $D(B,\sigma,c)$ is isomorphic to one of $D(B,\sigma,d)$ or $D_r(B,\sigma,d)$, by Corollary \ref{CDSBIsomorphisms 2} there must exist some $b\in F^{\times}$ such that $c=db^2$. This implies $d=cb^{-2}\in F^{\times}$, which is a contradiction.\\
Finally, if $D_m(B,\sigma,d)\cong D(B,\sigma,c)$, then the middle nuclei of the two algebras must be isomorphic; that is, $B\cong \lbrace k\in B\mid \sigma(k)d=d\sigma(k)\rbrace$. This is satisfied if and only if $d\in F^{\times}$, contradicting our assumption.
\end{proof}

Note that we cannot use an analogous proof to the one in Theorem \ref{CDSBIsomorphisms} to determine the isomorphisms of $D_m(B,\sigma,c)$, as the middle nucleus is not equal to $B$. We obtain some weaker results:

\begin{lemma}If $Fix(\sigma)\not\cong Fix(\phi)$, then $D_m(B,\sigma,c)\not\cong D_m(B',\phi,d)$ for any choice of $c\in B^{\times}$ and $d\in B'^{\times}$.
\end{lemma}
\begin{proof}
If $D_m(B,\sigma,c)\cong D_m(B',\phi,d)$, the left nucleus of $D_m(B,\sigma,c)$ is mapped isomorphically to the left nucleus of $D_m(B',\phi,d)$. By Lemma \ref{CSANucleusM}, this implies $Fix(\sigma)\cong Fix(\phi)$.
\end{proof}

\begin{theorem}\label{CDSBIsomorphismsR}
Let $D=D_m(B,\sigma,c)$ and $D'=D_m(B',\phi,d)$ be $F$-algebras. If $\tau:B\to B'$ is an $F$-isomorphism such that $\phi\circ\tau=\tau\circ\sigma$ and $\tau(c)=db^2$ for some $b\in F^{\times}$, there is an isomorphism $G:D\to D'$ given by $G(x,y)=(\tau(x),\tau(y)b)$ for all $x,y\in B$.
\end{theorem}
\begin{proof}
Clearly this is an $F$-vector space isomorphism from $B\oplus B$ to $B'\oplus B'$ as it is additive, bijective and $F$-linear. To show this map is multiplicative and thus an $F$-algebra isomorphism, we consider $G(u,v)G(x,y)=G((u,v)(x,y)).$ This is equivalent to $$\tau(u)\tau(x)+\phi(\tau(v)b)d\phi(\tau(y)b)=\tau(ux+\sigma(v)c\sigma(y)),\ \tau(u)\tau(y)b+\tau(v)b\tau(x)=\tau(uy+vx)b.$$ As $b\in F^{\times}$, this is equivalent to simply considering $\phi(\tau(v))db^2\phi(\tau(y))=\tau(\sigma(v))\tau(c)\tau(\sigma(y))$. Substituting $\tau(c)=db^2$, we conclude that this is satisfied for all $v,y\in B$ as we assumed $\phi\circ\tau=\tau\circ\sigma$. Hence $G:D\to D'$ is a $F$-algebra isomorphism.
\end{proof}


%

\subsection{Automorphisms}

\begin{theorem}\label{CDSBAutomorphisms}
Let $D=D(B,\sigma, c)$ (resp. $D=D_r(B,\sigma,c)$). All automorphisms $G:D\to D$ are of the form $$G(u,v)=(\tau(u),\tau(v)b)$$ for some $\tau\in Aut_F(B)$ such that $\tau\in C(\sigma)$ and $b\in F^{\times}$ satisfying $\tau(c)=cb^2$. Further, all maps of this form with $\tau\in Aut_F(B)$ and $b\in F^{\times}$ satisfying these conditions yield automorphisms of $D$.
\end{theorem}
\begin{proof}
Suppose that $G:D\to D$ is an $F$-automorphism. Then $G$ restricts to an automorphism of the middle nucleus of $D$. This means that $G$ restricted to $B$ must be an automorphism of $B$; that is, $G\!\mid_B=\tau\in Aut_F(B)$, so we have $G(x,0)=(\tau(x),0)$ for all $x\in B$.\\
Let $G(0,1)=(a,b)$ for some $a,b\in B$. Then we have $$G(x,y)=G(x,0)+G(0,1)G(y,0)=(\tau(x)+a\tau(y),\tau(y)b),$$ and $G(x,y)=G(x,0)+G(y,0)G(0,1)=(\tau(x)+\tau(y)a,b\tau(y)).$
This implies that $a,b\in Z(B')=F$.\\
As $G$ is multiplicative, we must also have $G((0,1)^2)=G(0,1)^2$ which holds if and only if $(a,b)(a,b)=(\tau(c),0).$ From this, we obtain the equations $a^2+c\phi(b^2)=\tau(c)$ and
$ab+ba=0.$ Since we have $a,b\in F$, this simplifies to $a^2+cb^2=\tau(c)$ and $2ab=0.$ As $F$ does not have characteristic 2, this implies either $a=0$ or $b=0$. If $b=0$, then $G(x,y)=(\tau(x)+\tau(y)a,0)$ and so $G$ is not surjective. This is a contradiction, as $G$ is an automorphism. Thus we conclude $a=0$ and $cb^2=\tau(c)$.\\
Finally, as $G$ is multiplicative we have $G(u,v)G(x,y)=G((u,v)(x,y))$ for all $u,v,x,y\in K$. When $D=D(B,\sigma,c)$, this yields $$(\tau(ux)+c\sigma(\tau(v)b\tau(y)b),\tau(uy)b+\tau(v)b\tau(x))=(\tau(ux+c\sigma(vy)),\tau(uy+vx)b)$$ for all $u,v,x,y\in K$. As $b\in F$, this implies we must have $cb^2\sigma(\tau(vy))=\tau(c\sigma(vy)).$ After substituting the condition $\tau(c)=c\phi(b^2)$, we get $\sigma\circ\tau=\tau\circ\sigma.$ This follows almost identically for $D_r(B,\sigma,c)$.\\
Conversely, let $G:B\oplus B\to B\oplus B$ be defined by $G(x,y)=(\tau(x),\tau(y)b)$ for some $F$-automorphism $\tau:B\to B$ such that $\sigma\circ\tau=\tau\circ\sigma$ and $\tau(c)=cb^2$ for some $b\in F^{\times}$. It is easily checked that this in fact gives an $F$-algebra automorphism of $D$.
\end{proof}

\begin{corollary}
Let $D=D(B,\sigma, c)$ (resp. $D=D_r(B,\sigma,c)$). There is a subgroup of $Aut_F(D)$ isomorphic to $$\lbrace\tau\in Aut_F(B)\mid \tau(c)=c \mbox{ and } \tau\circ\sigma=\sigma\circ\tau\rbrace.$$
\end{corollary}

In order to describe the number of automorphisms of $D(B,\sigma,c)$ and $D_r(B,\sigma,c)$, we introduce a slightly different version of the group $J(c)$: $$J_F(c)=\lbrace \tau\in Aut_F(B)\mid X^2-\tau(c)c^{-1} \mbox{ has solutions in }F\rbrace\subset Aut_F(B).$$ Similarly to $J(c)$, this forms a subgroup of $Aut_F(B)$. The proof of this follows identically to the proof of Theorem \ref{JcSubgroup}.

\begin{theorem}
There are exactly $2\left|J_F(c)\cap C(\sigma)\right|$ automorphisms of $D(B,\sigma,c)$ (respectively $D_r(B,\sigma,c)$), each of which is given by the automorphisms $G(x,y)=(\tau(x),\tau(y)b_i)$ for each $\tau\in J_F(c)\cap C(\sigma)$, where $b_i\in F$ are the two solutions of $X^2-\tau(c)c^{-1}$ for $i=1,2$.
\end{theorem}
\begin{proof}
The proof follows analogously to the proof of Theorem \ref{CDSExactNumberOfAutomorphisms}, apart from requiring that $b_i\in F^{\times}$. This is due to the constraints determined in Theorem \ref{CDSBAutomorphisms}.
\end{proof}

\begin{corollary} If $c\in F^{\times}$, then there are exactly $2\left|C(\sigma)\right|$ automorphisms of $D(B,\sigma,c)$, each of which is given by the automorphisms $G(x,y)=(\tau(x),\pm\tau(y))$ for each $\tau\in C(\sigma)$.
\end{corollary}
\begin{proof}
This follows similarly to Corollary \ref{CDSCSigmaAuts}.
\end{proof}

An integral part of the proof given in Theorem \ref{CDSBAutomorphisms} is that one of the nuclei of these algebras must be equal to $B$ and so any automorphism of $D(B,\sigma,c)$ must restrict to an automorphism of $B$. For $D_m(B,\sigma,c)$ with $c\not\in F^{\times}$, $B$ is not equal to any of the nuclei so we cannot make this deduction. However, if we assume that an automorphism of $D_m(B,\sigma, c)$ restricts to an automorphism of $B$, then it must be of the same form as the automorphisms of the other Dickson algebras:

\begin{theorem}
Let $D=D_m(B,\sigma,c)$ and suppose $G$ is an automorphism which restricts to an automorphism of $B$. Then $$G(u,v)=(\tau(u),\tau(v)b)$$ for some $\tau\in Aut_F(B)$ such that $\tau\in C(\sigma)$ and $b\in F^{\times}$ satisfying $\tau(c)=cb^2$.
\end{theorem}
\begin{proof}
The proof follows analogously to Theorem \ref{CDSBAutomorphisms} as $G$ restricts to an automorphism of $B$.
\end{proof}

\printbibliography

\end{document}